\documentclass[12pt,reqno]{amsproc}
  \usepackage{latexsym} 
  \usepackage[all]{xy}
  \usepackage{amsfonts} 
  \usepackage{amsthm} 
  \usepackage{amsmath} 
  \usepackage{amssymb}
  \usepackage{pifont}  
  \usepackage{enumerate}
  \xyoption{2cell}

 \usepackage{tikz}
\usetikzlibrary{arrows}

  \def\sw#1{{\sb{(#1)}}} 
   
  \def\su#1{{\sp{[#1]}}}  
  \def\suc#1{{\sp{(#1)}}}

  \def\<{{\langle}} 
  \def\>{{\rangle}}

  \def\eps{\varepsilon} 

  \def\note#1{{}}

  \def\note#1{} 

  \def\cM{{\mathbf{cM}}} 
  \def\cuM{{\mathbf{cuM}}}

  \def\beq{\begin{equation}} 
  \def\eeq{\end{equation}}

  \def\id{\mathrm{id}}

  \def\ot{{\otimes}}

  \def\roM{\varrho^{M}} 
   \def\roN{\varrho^{N}}

  \newcounter{zlist} 
  \newenvironment{zlist}{\begin{list}{(\arabic{zlist})}{ 
  \usecounter{zlist}\leftmargin2.5em\labelwidth2em\labelsep0.5em 
  \topsep0.6ex
  \parsep0.3ex plus0.2ex minus0.1ex}}{\end{list}}

  \newcounter{blist} 
  \newenvironment{blist}{\begin{list}{(\alph{blist})}{ 
  \usecounter{blist}\leftmargin2.5em\labelwidth2em\labelsep0.5em 
  \topsep0.6ex 
  \parsep0.3ex plus0.2ex minus0.1ex}}{\end{list}} 

  \newcounter{rlist}

\def\stac#1{\raise-.2cm\hbox{$\stackrel{\displaystyle\otimes}{\scriptscriptstyle{#1}}$}}

\def\cten#1{\raise-.2cm\hbox{$\stackrel{\displaystyle\widehat{\otimes}}
{\scriptscriptstyle{#1}}$}}

  \def\Label#1{\label{#1}\ifmmode\llap{[#1] }\else 
  \marginpar{\smash{\hbox{\tiny [#1]}}}\fi} 
  \def\Label{\label}

  \newtheorem{proposition}{Proposition}[section]
  \newtheorem{lemma}[proposition]{Lemma} 
  \newtheorem{corollary}[proposition]{Corollary}

\theoremstyle{definition} 
  \newtheorem{definition}[proposition]{Definition}
    
  \newtheorem{example}[proposition]{Example}

  \theoremstyle{remark} 
  \newtheorem{remark}[proposition]{Remark}

  \newcounter{c} 
   
  \newcommand{\etyk}[1]{\vspace{-7.4mm}$$\begin{equation}\Label{#1} 
  \addtocounter{c}{1}} 
  \renewcommand{\]}{\ifnum \value{c}=1 $$\else \end{equation}\fi} 
\numberwithin{equation}{section}

\def\ot{\otimes}

\def\KK{{\mathbb K}}

\newcommand{\Cc}{\mathcal{C}}

\newcommand{\Tt}{\mathcal{T}}

\def\M{{\bf M}}

\def\*C{{}^*\hspace*{-1pt}{\Cc}}

\def\text#1{{\rm {\rm #1}}}

 \def\1{\mathbf{1}}

  \def\uM{\mathbf{uM}}

\begin{document}

\title[Rota-Baxter systems]{Rota-Baxter systems, dendriform algebras and covariant bialgebras}

\author{Tomasz Brzezi\'nski}
 \address{ Department of Mathematics, Swansea University, 
  Swansea SA2 8PP, U.K.} 
  \email{T.Brzezinski@swansea.ac.uk}   
 
\begin{abstract}
A generalisation of the notion of a Rota-Baxter operator is proposed. This generalisation consists of two operators acting on an associative algebra and satisfying equations similar to the Rota-Baxter equation. Rota-Baxter operators of any weights and twisted Rota-Baxter operators are solutions of the proposed system. It is shown that dendriform algebra structures of a particular kind are equivalent to Rota-Baxter systems. It is shown further that a Rota-Baxter system induces a {\em weak peudotwistor} [F.\ Panaite \& F.\ Van Oystaeyen, {\em Twisted algebras, twisted bialgebras and Rota-Baxter operators,} arXiv:1502.05327 (2015)] which can be held responsible for the existence of a new associative product on the  underlying algebra. Examples of solutions of Rota-Baxter systems are obtained from quasitriangular covariant bialgebras hereby introduced as a natural extension of {\em infinitesimal bialgebras} [M.\ Aguiar, {\em Infinitesimal Hopf algebras,} [in:]  {\em New trends in Hopf algebra theory (La Falda, 1999)}, Contemp.\ Math., {\bf 267}, Amer.\ Math.\ Soc., Providence, RI, (2000), pp.\ 1--29].
\end{abstract}
\maketitle

\section{Introduction}
This paper arose form an attempt to understand the Jackson $q$-integral as a Rota-Baxter operator, and develops and extends connection between three algebraic systems: Rota-Baxter algebras \cite{Rot:Bax}, dendriform algebras \cite{Lod:dia} and infinitesimal bialgebras \cite{Agu:inf}. 

Given an associative algebra $A$ over a field $\KK$ and $\lambda\in \KK$, a linear operator $R:A\to A$ is called a {\em Rota-Baxter operator of weight $\lambda$} if, for all $a,b\in A$,
\begin{equation}\label{RB.lambda}
R(a)R(b) = R\left(R(a) b + a R(b) +\lambda ab\right).
\end{equation} 
In this case the triple $(A,R,\lambda)$ is referred to as a {\em Rota-Baxter algebra of weight $\lambda$}. Rota-Baxter operators were introduced in \cite{Bax:ana} in the context of differential operators on commutative Banach algebras and since \cite{Rot:Bax} intensively studied in probability and combinatorics, and more recently in the theory of operads and renormalisation of quantum field theories. 

Recall from \cite[Section~5]{Lod:dia} that a {\em dendriform algebra} is a system consisting of a vector space $V$ and two bilinear operations $\prec,\succ$ on $V$ such that, for all $a,b,c\in V$,
\begin{subequations}\label{dend}
\begin{gather}
(a\prec b)\prec c = a\prec (b\prec c + b\succ c) \label{dend.1}\\
a\succ (b\prec c) = (a\succ b)\prec c \label{dend.2}\\
a\succ (b\succ c) =(a\prec b + a\succ b)\succ c \label{dend.3}
\end{gather}
\end{subequations}
As explained in \cite{Agu:pre}, \cite{Ebr:Lod}, \cite{EbrMan:den} every Rota-Baxter operator of weight $\lambda$ on an algebra $A$ defines a dendriform algebra structure on $A$ in a variety of ways including 
\begin{equation}\label{RB.dend}
a\succ  b := R(a) b, \qquad a\prec b := a(R(b) +\lambda b).
\end{equation}

An associative (not necessarily unital) algebra $A$ that admits a coassociative comultiplication which is a derivation is called an {\em infinitesimal bialgebra} \cite{Agu:inf}. That is, in addition to the coassociative law, the comultiplication $\Delta:A\to A\ot A$ satisfies 
\begin{equation}\label{deriv}
\Delta(ab) = a\Delta(b) + \Delta(a)b, \qquad \mbox{for all $a,b\in A$,}
\end{equation}
where $A\ot A$ is viewed as an $A$-bimodule in the standard way $a\cdot (b\ot c) \cdot d = ab\ot cd$. As shown in \cite{Agu:inf}, if an element $r\in A\ot A$ satisfies the associative classical Yang-Baxter equation, 
\begin{equation}\label{YB.ass}
r_{13}r_{12} - r_{12}r_{23} +r_{23}r_{13} = 0
\end{equation}
(see \eqref{legs} below for the explanation of the index notation used),
then the inner derivation induced by $r$ (i.e. a commutator with $r$) is coassociative, and thus defines on $A$ the structure of an infinitesimal bialgebra (called a {\em quaistriangular infinitesimal bialgebra}). Furthermore, it is proven  in \cite{Agu:info} that every solution of the associative classical Yang-Baxter equation defines a Rota-Baxter operator of weight $0$.

In this article we propose to study two operators $R,S$, both acting on the same associative algebra $A$, that satisfy equations similar to the Rota-Baxter equation \eqref{RB.lambda}  with $\lambda =0$. We show that in the case of a non-degenerate algebra (see the explanation below) this system of equations is equivalent to the existence of a dendriform algebra structure of the type given in \eqref{RB.dend} (with $R(b) + \lambda b$ replaced by $S(b)$). A class of solutions to Rota-Baxter systems arise from {\em associative Yang-Baxter pairs}. These are defined as pairs of elements $r,s\in A\ot A$, which satisfy two equations similar to the classical associative Yang-Baxter equation \eqref{YB.ass}: each equation involves both $r$ and $s$ albeit not in a symmetric way. In order to give a conceptual grounding for associative Yang-Baxter pairs we relate them to {\em covariant bialgebras}. In this class of algebras the coproduct is no longer assumed to be a derivation (as is the case for infinitesimal bialgebras) but a covariant derivation (or a connection) with respect to a pair of derivations. This relaxing of the definition of an infinitesimal bialgebra allows us to include other types of algebras characterised by the existence of a coproduct, such as Frobenius algebras. We give a characterisation and some examples of bicovariant bialgebras, most significantly we show that an associative Yang-Baxter pair defines a covariant bialgebra, termed a {\em quasitriangular covariant bialgebra}. We define and analyse basic properties of the representation category of a covariant bialgebra, termed the category of {\em covariant modules}. Finally, we introduce the notion of an {\em endomorphism twisted Rota-Baxter operator}, and observe that such an operator induces a Rota-Baxter system, and that the Jackson's $q$-integral is an example of a twisted Rota-Baxter operator.

 All algebras considered in this paper are associative (but not necessarily unital)  over a field $\KK$. We say that an algebra $A$ is {\em non-degenerate} provided that for any $b\in A$,  if $ba=0$ or $ab=0$ for all $a\in A$ implies that $b=0$. Obviously, any unital algebra is non-degenerate. Given $r = \sum_i a_i\ot b_i \in A\ot A$, we define
\begin{equation}\label{legs}
r_{12} = 1\ot r, \qquad r_{13} = \sum_i a_i\ot 1\ot b_i, \qquad r_{23} = 1\ot r,
\end{equation}
where $1$ means either the identity of $A$ (if $A$ is unital) or the identity in the extended unital algebra $\KK\oplus A$ (if $A$ is non-unital).

\section{Rota-Baxter systems and dendriform algebras}
In this section we define Rota-Baxter systems and relate them to dendriform algebras.

\label{sec.R-B}\setcounter{equation}{0}
\begin{definition}\label{def.R-B}
A triple $(A,R,S)$ consisting of an algebra $A$ and two $\KK$-linear operators $R,S:A\to A$ is called a {\em Rota-Baxter system} if, for all $a,b\in A$,
\begin{subequations}\label{RB}
\begin{gather}
R(a)R(b) = R\left(R(a) b + a S(b)\right) , \label{RB.r}\\
S(a)S(b) = S\left(R(a) b + a S(b)\right) . \label{RB.s}
\end{gather}
\end{subequations}
\end{definition}

Rota-Baxter operators \eqref{RB.lambda} together with algebras on which they operate are examples of Rota-Baxter systems as explained in the following
\begin{lemma}
Let $A$ be an algebra.
 If $R$ is a Rota-Baxter operator of weight $\lambda$ on $A$, then $(A,R,R+\lambda\id)$ and $(A,R+\lambda\id, A)$ are Rota-Baxter systems.
\end{lemma}
\begin{proof} The replacement of $S$ by $R+\lambda\id$ in equation \eqref{RB.r} gives precisely the Rota-Baxter relation \eqref{RB.lambda}. An elementary calculation then reveals that in this case \eqref{RB.s} reproduces \eqref{RB.r}. The second claim follows by the  fact that
$$
(R(a) +\lambda a) b + aR(b) = R(a) b + a(R(b)+\lambda b),
$$
compiled with the $R$-$S$-symmetry of Definition~\ref{def.R-B}.
\end{proof}

\begin{lemma}\label{lem.linear}
Let $A$ be an algebra, $R:A\to A$ a left $A$-linear map and $S:A\to A$ a right $A$-linear map. Then $(A,R,S)$ is a Rota-Baxter system if and only if, for all $a,b\in A$,
\begin{equation}\label{RB.linear}
aR\circ S(b) = 0 = S\circ R(a) b.
\end{equation}
In particular, if $A$ is a non-degenerate algebra, then  $(A,R,S)$ is a Rota-Baxter system (with $A$ left and $S$ right $A$-linear) if and only if $R$ and $S$ satisfy the orthogonality condition
\begin{equation}\label{RB.linear.non}
R\circ S = S\circ R =0.
\end{equation}
\end{lemma}
\begin{proof}
If $(A,R,S)$ is a Rota-Baxter system, then, by the left $A$-linearity of $R$, for all $a,b\in A$,
$$
R(a)R(b) = R\left(R(a) b + a S(b)\right) = R(a)R(b) + aR(S(b)),
$$
hence the first of equations \eqref{RB.linear}. In a similar way using the right $A$-linearity of $S$ and $\eqref{RB.s}$ one obtains the second of equations \eqref{RB.linear}. Tracing the above steps backwards one immediately realises that \eqref{RB.linear} implies \eqref{RB}.

The second statement is a straightforward consequence of the first one and the definition of a non-degenerate algebra.
\end{proof}

\begin{example}\label{ex.linear}
Suppose that $r,s\in A$ are such that $rs=0$, with one of them, say $s$, being central. Define  $R,S:A\to A$ by
$$
R: a\mapsto ar, \qquad S: a\mapsto sa.
$$
Then, for all $a,b\in A$,
$$
aR(S(b)) = asbr = abrs = 0, \qquad S(R(a))b = sarb = arsb =0,
$$
by the centrality of $s$ and since $rs=0$. Hence $(A,R,S)$ is a Rota-Baxter system.

As a specific example, take a truncated polynomial algebra $A=\KK[\zeta]/\langle \zeta^n\rangle$, and let $(p,q)$ be a partition of $n$. Define $R_p(a) = a\zeta^p$, $S_q(a) = a\zeta^q$. Then $(A,R_p,S_q)$ is a Rota-Baxter system.

For a more geometric example, consider $A$ to be a coordinate algebra of the algebraic variety consisting of two straight lines crossing at one point, and $R$ and $S$ to arise from projections on the first and the second line respectively. Then $(A,R,S)$ is a Rota-Baxter system.
\end{example}

\begin{proposition}\label{prop.dend}
Let $A$ be an associative algebra and let $R,S:A\to A$ be $\KK$-linear homomorphisms. Define $\KK$-linear maps $\prec,\succ: A\ot A\to A$, by
\begin{equation}\label{oper.dend}
a \prec b = aS(b), \qquad a\succ b = R(a)b, \qquad \mbox{for all $a,b\in A$.}
\end{equation}
Then
\begin{zlist} 
\item If $(A,R,S)$ is a Rota-Baxter system then  $(A, \prec, \succ)$ is a dendriform algebra.
\item If $A$ is non-degenerate and $(A, \prec, \succ)$ is a dendriform algebra, then $(A,R,S)$ is a Rota-Baxter system.
\end{zlist}
\end{proposition}
\begin{proof}
(1) Assume that  $(A,R,S)$ is a Rota-Baxter system. Then, for all $a,b,c\in A$,
$$
(a\prec b)\prec c = aS(b)S(c) = aS\left(R(b) c + b S(c)\right) = a\prec (b\prec c + b\succ c),
$$
by \eqref{RB.s}. In a similar way \eqref{RB.r} implies \eqref{dend.3}. Finally, \eqref{dend.2} follows by  the associativity of $A$.

(2) In the converse direction, assume that  $(A, \prec, \succ)$, with $\prec, \succ$ given by \eqref{oper.dend} is a dendriform algebra. Then  the dendriform relation \eqref{dend.3} comes out as 
$$
\left(R(a)R(b) - R\left(R(a) b + a S(b)\right)\right)c =0,
$$
and hence it  gives \eqref{RB.r} by the non-degeneracy of the product in $A$. In a similar manner,  \eqref{dend.1} gives \eqref{RB.s}.
\end{proof}

\begin{remark}\label{rem.cat}
A morphism of Rota-Baxter systems from $(A,R_A,S_A)$ to $(B,R_B,S_B)$ is an algebra map $f:A\to B$ rendering the following diagrams commutative:
$$
\xymatrix{A\ar[rr]^{R_A}\ar[d]_{f} & & A\ar[d]^f\\ B\ar[rr]^{R_B} & & B,} \qquad \xymatrix{A\ar[rr]^{S_A}\ar[d]_{f} & & A\ar[d]^f\\ B\ar[rr]^{S_B} & & B.}
$$
The assignment of a dendriform algebra to a Rota-Baxter system described in Proposition~\ref{prop.dend} defines a faithful functor from the category of Rota-Baxter systems to the category of dendriform algebras.
\end{remark}

\begin{corollary}\label{cor.RB}
Let $(A,R,S)$ be a Rota-Baxter system. Then
\begin{zlist}
\item $(A,*)$ with $*: A\ot A\to A$ defined by
\begin{equation}\label{prod*}
a*b = R(a)b+aS(b), \qquad \mbox{for all $a,b\in A$,}
\end{equation}
is an associative algebra.
\item $(A, \bullet)$ with $\bullet:A\ot A\to A$,
\begin{equation}\label{prod.}
a\bullet b = R(a)b-bS(a), \qquad \mbox{for all $a,b\in A$,}
\end{equation}
is a pre-Lie  algebra.
\end{zlist}
\end{corollary}
\begin{proof}
Since $a\ot b \mapsto R(a)b$ and $a\ot b\mapsto aS(b)$ are dendriform operations, assertion (1) follows by \cite[5.2~Lemma]{Lod:dia} and (2) follows by \cite[Lemma~13.6.4]{LodVal:alg}.
\end{proof}

Another way of understanding the associativity of product \eqref{prod*} is by connecting Rota-Baxter systems with recently introduced {\em weak pseudotwistors} \cite{PanVan:twi}.
\begin{definition}\label{def.twistor}
Let $A$ be an algebra with associative product $\mu:A\ot A\to A$. A $\KK$-linear map $T:A\ot A\to A\ot A$ is called a {\em weak pseudotwistor} if there exists a $\KK$-linear map $\Tt:A\ot A\ot A\to A\ot A\ot A$, rendering commutative the following diagram:
\begin{equation}\label{bow-tie}
\xymatrix{
&A\ot  A\ot  A 
\ar[r]^-{\id \otimes \mu} & A\ot  A 
\ar[dd]^{T}  & A\ot  A\ot  A \ar[l]_-{\mu \ot \id} &\\
A\ot A\ot  A \ar[ur]^{\id \ot T} \ar[dr]_{\Tt} & & &  & A\ot  A\ot  A \ar[ul]_{T\ot \id}
\ar[dl]^{\Tt}\\
&   A\ot  A\ot  A \ar[r]_-{\id \ot \mu} &  A\ot  A & A\ot A\ot  A. \ar[l]^-{\mu\ot\id}  &}
\end{equation}
The map $\Tt$ is called a {\em weak companion} of $T$.
\end{definition}
\begin{lemma}\label{lem.twistor}
Let $(A,R,S)$ be a Rota-Baxter system. Then 
\begin{equation}\label{RBtwistor}
T: A\ot A\to A\ot A, \qquad a\ot b \mapsto R(a)\ot b + a\ot S(b),
\end{equation}
is a weak pseudotwistor.
\end{lemma}
\begin{proof}
Define a $\KK$-linear map:
$
\Tt: A\ot A\ot A\to A\ot A\ot A, 
$
by
\begin{equation}\label{comp}
 \Tt(a\ot b\ot c) =R(a)\ot R(b)\ot c + R(a) \ot b \ot S(c) + a \ot S(b)\ot S(c).
\end{equation}
Then
\begin{eqnarray*}
(\mu\ot \id) \circ \Tt(a\ot b \ot c) &=& R(a)R(b)\ot c + R(a)b \ot S(c) + aS(b)\ot S(c)\\
&=& R(R(a)b)\ot c + R(aS(b)) \ot c \\
&& + R(a)b \ot S(c) + aS(b)\ot S(c)\\
&=& T\circ (\mu\ot \id) \circ(T\ot \id)(a\ot b \ot c),
\end{eqnarray*}
by \eqref{RB.r}. Thus the right pentagon in \eqref{bow-tie} is commutative. Similarly, \eqref{RB.s} renders the left pentagon commutative, and $T$ is a pseudotwistor with companion $\Tt$.
\end{proof}

Note that the product $*$ defined by \eqref{prod*} is simply equal to $\mu\circ T$, where $T$ is given in \eqref{RBtwistor}. Hence $*$ is associative by \cite[Theorem~2.3]{PanVan:twi}. Also note that, similarly to the case of dendriform algebras, if $A$ is non-degenerate, then $T$ given by \eqref{RBtwistor} is a weak pseudotwistor with the companion \eqref{comp} if and only if $(A,R,S)$ is a Rota-Baxter system.

The remainder of this article is devoted to presentation of examples of Rota-Baxter systems and to placing them within a comprehensive algebraic framework.

\section{Covariant bialgebras}
\label{sec.cov}
This section is divided into four parts. In the first part a system of equations is given, whose solution leads to a Rota-Baxter system. In the second, the system introduced in the first is given a more conceptual grounding, based on the ideas developed in \cite{Agu:inf}. In the third part a representation category for covariant bialgebras  introduced in the second one is studied. Finally, in the fourth part some comments on extensions of the results of the first two parts to the non-commutative base are made.

\subsection{Associative Yang-Baxter pairs}\label{sec.rs}
\begin{definition}\label{def.rs}
Let $A$ be an associative algebra. An {\em associative  Yang-Baxter pair} is a pair of elements $r,s \in A\ot A$ that satisfy the following equations
\begin{subequations}\label{rs}
\begin{gather}
r_{13}r_{12} - r_{12}r_{23} +s_{23}r_{13} = 0, \label{rs.r}\\
s_{13}r_{12} - s_{12}s_{23} +s_{23}s_{13} = 0, \label{rs.s}
\end{gather}
\end{subequations}
where $r_{12} = r\ot 1$, $r_{23} = 1\ot r$, etc., see \eqref{legs}.
\end{definition}

\begin{lemma}\label{lem.homo}
Let $f:A \to B$ be an algebra homomorphism. If $(r,s)$ is an associative Yang-Baxter pair in $A$, then 
\begin{equation}\label{rfsf}
r^f := (f\ot f)\circ r \quad \mbox{and}  \quad  s^f:=(f\ot f)\circ s
\end{equation}
 form an associative Yang-Baxter pair in $B$.
\end{lemma}
\begin{proof}
This follows immediately from the multiplicativity of $f$.
\end{proof}

\begin{example}\label{ex.YB}~

(1) The couple $(r,r)$ is an associative Yang-Baxter pair if and only if $r$ is a solution to the classical associative Yang-Baxter equation \eqref{YB.ass}.

(2) If $r\in A\ot A$ is a solution to the {\em Frobenius-separability} or {\em FS-equation},
\begin{equation}\label{FS}
r_{12}r_{23} = r_{23}r_{13} = r_{13}r_{12},
\end{equation}
\cite[Lemma~3.2]{BeiFon:Fro}, \cite[Section~8.2]{CaeMil:Fro}, then $(r,0)$ and $(0,r)$ are associative Yang-Baxter pairs.

(3) If $r,s \in A\ot A$ are solutions to the FS-equation \eqref{FS} such that \begin{equation}\label{orto}
s_{23}r_{13} = s_{13}r_{12} = 0,
\end{equation}
 then $(r,s)$ is an associative Yang-Baxter pair.

For a specific example, let $A= M_m(\KK)\oplus M_n(\KK)$ be the direct sum of matrix algebras thought of as block-diagonal matrices in $M_{m+n}(\KK)$. Fix $k\in \{1,\ldots , m\}$ and $l\in \{m+1,\ldots , m+n\}$ and define
\begin{equation}\label{rs.matrix}
r = \sum_{i=1}^m e_{ik}\ot e_{ki}, \qquad s = \sum_{j=m+1}^{m+n} e_{jl}\ot e_{lj},
\end{equation}
where   $e_{ij}$ are the matrices with $1$ in the $(i,j)$-th entry and  $0$ elsewhere. Then  $(r,s)$ is an associative Yang-Baxter pair.

(4) Examples (2) and (3) can be modified by splitting the FS-equation into two equations
\begin{equation}\label{as.coa}
 r_{13}r_{12} = r_{12}r_{23}, \qquad s_{12}s_{23} = s_{23}s_{13}
 \end{equation}
 If $r,s$ satisfy \eqref{as.coa}, then both $(r,0)$ and $(0,s)$ are associative Yang-Baxter pairs. If, in addition, $r,s$ satisfy the orthogonality condition \eqref{orto}, then$(r,s)$ is an associative Yang-Baxter pair.
 
As explained in \cite[Section~8.2]{CaeMil:Fro}, the first of equations \eqref{as.coa} can be interpreted as an associativity condition for a particular multiplication defined on $V\otimes V$, where $V$ is a finite dimensional vector space. Dually, the second of equations \eqref{as.coa} can be interpreted as a coassociativity condition for a particular comultiplication defined on a unital associative algebra \cite[Propositon~153]{CaeMil:Fro}. From a different perspective the first of conditions \eqref{as.coa} can be viewed  as the associative version of the integrability or flatness condition for the Knizhnik-Zamolodchikov connection; see e.g.\ \cite[Chapter~XIX]{Kas:gro} or \cite[Chapter~12]{ShnSte:gro}.

(5) Suppose that $A$ contains $g,h$ such that $g^2 =0$ and $gh=hg=0$. Then $r=g\ot h$, $s=h\ot g$ is an associative Yang-Baxter pair. 
\end{example}

\begin{proposition}\label{prop.YBpair}
Let $(r,s)$ be an associative Yang-Baxter pair in $A$. Write 
\begin{equation}\label{sweed}
r = \sum r\su 1\ot r\su 2, \qquad s = \sum s\su 1\ot s\su 2,
\end{equation}
(summation indices suppressed), and define
\begin{equation}\label{RS.YB}
R,S: A\to A, \qquad R(a) = \sum r\su 1a r\su 2, \quad S(a) = \sum s\su 1a s\su 2.
\end{equation}
Then $(A,R,S)$ is a Rota-Baxter system.

Furthermore, if $f:A\to B$ is an algebra map and $(r^f,s^f)$ is the associative Yang-Baxter pair induced by $f$ as in \eqref{rfsf}, then $f$ is a morphism of Rota-Baxter systems from the system associated to $(r,s)$ to the system associated to $(r^f,s^f)$.
\end{proposition}
\begin{proof}
In terms of the Sweedler-like notation \eqref{sweed}, equations \eqref{rs} come out as
\begin{subequations}\label{rs.sweed}
\begin{gather}
\sum r\su 1\tilde{r}\su 1\ot \tilde{r}\su 2\ot r\su 2 - \sum r\su 1\ot r\su 2\tilde{r}\su 1\ot \tilde{r}\su2 + \sum r\su 1\ot s\su 1 \ot s\su 2r\su 2 =0,\\
\sum s\su 1r\su 1\ot  r\su 2\ot s\su 2 - \sum s\su 1\ot s\su 2\tilde{s}\su 1\ot \tilde{s}\su2 + \sum  \tilde{s}\su 1\ot s\su 1 \ot s\su 2\tilde{s}\su 2 =0,
\end{gather}
\end{subequations}
where $\sum \tilde{r}\su 1\ot \tilde{r}\su 2$,  $\sum \tilde{s}\su 1\ot \tilde{s}\su 2$ denote another copies of $r$ and $s$ respectively. Replacing tensor products in \eqref{rs.sweed} by $a$ and $b$ and using definition \eqref{RS.YB} one obtains equations \eqref{RB}, as required.

The second statement follows by the definition of $(r^f,s^f)$ and by the multiplicativity of $f$.
\end{proof}

\begin{example}\label{ex.YB.RS}~

(1) In the case of matrices of Example~\ref{ex.YB}~(3), the Rota-Baxter system associated to \eqref{rs.matrix} consists of the operator $R$, which acting on a (block-diagonal) matrix $(a_{ij})$ returns a matrix with entries $a_{kk}$ on the first $m$ diagonal places and zeros elsewhere. $S$ returns a matrix containing $a_{ll}$ in the last $n$ diagonal places as only possible non-zero entries.

(2) In the case of Example~\ref{ex.YB}~(4), the Rota-Baxter system associated to $(r,s)$ satisfying \eqref{as.coa} and \eqref{orto} will  satisfy the separated equations
$$
R(a)R(b) = R(R(a)b), \quad S(a)S(b) = S(aS(b)), \qquad S(R(a)b) = R(aS(b)) =0.
$$

(3) In the setup of Example~\ref{ex.YB}~(5), the Rota-Baxter system associated to $g$ and $h$ is
$
R(a) = gah$, $S(a)= hag.
$
\end{example}

\subsection{Covariant bialgebras}
In \cite{Agu:inf} associative classical Yang-Baxter operators were connected with infinitesimal bialgebras, i.e.\ algebras admitting a coassociative coproduct that is a derivation. Following the same line of ideas, associative Yang-Baxter pairs are related to {\em covariant bialgebras} in which coproduct is required to be a covariant derivation. The aim of this section is to introduce covariant bialgebras and to reveal this relation.
\begin{definition}\label{def.cov.der}
Let $A$ be an associative algebra and $\delta_1,\delta_2 : A\to A\ot A$ derivations (cf.\ \eqref{deriv}). 
\begin{zlist}
\item If $M$ is a right $A$-module, then a  $\KK$-linear map $\nabla:M\to M\ot A$ is called a {\em right covariant derivation} (or a {\em right connection}) with respect to $\delta_1$ if
\begin{equation}\label{left.cov}
\nabla(ma) = \nabla(m)a + m\delta_1(a), \qquad \mbox{for all $a\in A$ and $m\in M$;}
\end{equation}
\item If $M$ is a left $A$-module, then a  $\KK$-linear map $\nabla:M\to A\ot M$ is called a {\em left covariant derivation} (or a {\em left connection}) with respect to $\delta_2$ if
\begin{equation}\label{right.cov}
\nabla(am) = a\nabla(m) + \delta_2(a)m, \qquad \mbox{for all $a\in A$ and $m\in M$;}
\end{equation}
\item A $\KK$-linear map $\nabla:A\to A\ot A$ is called a {\em covariant derivation}  with respect to $(\delta_1,\delta_2)$ if it is a right covariant derivation with respect to $\delta_1$ and left covariant derivation with respect to $\delta_2$, i.e.\
\begin{equation}\label{cov}
\nabla(ab) = \nabla(a)b + a\delta_1(b) = a\nabla(b) + \delta_2(a)b,, \qquad \mbox{for all $a,b\in A$.}
\end{equation}
\end{zlist}
\end{definition}

Obviously, any derivation $\delta: A\to A\ot A$ is a covariant derivation with respect to $(\delta,\delta)$, and a covariant derivation with respect to $(0,0)$ is the same as an $A$-bimodule map $A\to A\ot A$. Note also that if $A$ is a unital algebra, then a  covariant derivation is fully determined by its value at the identity of $A$.

\begin{definition}\label{def.cov.bialg}
A {\em covariant bialgebra} is a quadruple $(A,\delta_1,\delta_2,\Delta)$, such that 
\begin{blist}
\item $A$ is an associative algebra,
\item $\delta_1,\delta_2 : A\to A\ot A$ are derivations,
\item $(A,\Delta)$ is a coassociative coalgebra such that $\Delta$ is a  covariant derivation with respect to $(\delta_1,\delta_2)$.
\end{blist}
If $A$ has identity, then a covariant bialgebra  $(A,\delta_1,\delta_2,\Delta)$ is said to be {\em unital} provided $\Delta(1) =1\ot 1$.

A morphism of covariant bialgebras is a $\KK$-linear map that is both an algebra and coalgebra map.
\end{definition}

\begin{definition}\label{def.connected}
Let $(A,\delta_1,\delta_2, \Delta)$ be a covariant bialgebra. The subalgebra
\begin{equation}\label{constant}
C(A) := \ker \delta_1 \cap \ker \delta_2,
\end{equation}
is called a {\em constant subalgebra} of $(A,\delta_1,\delta_2, \Delta)$. 
If $A$ has identity, then we say that $(A,\delta_1,\delta_2, \Delta)$ is {\em left-connected} if $\ker \delta_1 =\KK 1$. Symmetrically, $(A,\delta_1,\delta_2, \Delta)$ is {\em right-connected} if $\ker \delta_2 =\KK 1$, and it is {\em connected} if it is both left- and right-connected.
\end{definition}

Since $\Delta$ is a covariant derivation with respect to $(\delta_1,\delta_2)$, $\Delta$ is a $C(A)$-bilinear map. In the case of unital $A$, for all $a\in \KK$, $\delta_1(a 1) = \delta_2(a 1) =0$. Therefore, $(A,\delta_1,\delta_2, \Delta)$ is right- or left-connected if and only if $C(A)= \KK 1$.

\begin{example}\label{ex.cov}~

(1) Recall that a unital algebra $A$ is said to be {\em Frobenius} if it is isomorphic to its vector space dual as a right (equivalently, left) $A$-module. By \cite[Proposition~2.1]{Abr:mod} a unital algebra $A$ is a Frobenius algebra if and only if it admits a coassociative and counital   comultiplication $\Delta$ that is a morphism of $A$-bimodules. Thus, if $A$ is a Frobenius algebra, then $(A,0,0,\Delta)$ is a covariant bialgebra. Obviously, the constant subalgebra of $(A,0,0,\Delta)$ is equal to $A$.

(2) $(A,\Delta,\Delta,\Delta)$ is a covariant bialgebra if and only if $(A,\Delta)$ is an infinitesimal bialgebra \cite{Agu:inf}, i.e.\ an algebra equipped with a coassociatitive comultiplication which is also a derivation. Obviously, $C(A) = \ker \Delta$.
\end{example}

\begin{proposition}\label{prop.cov.unit}
Let $A$ be a unital algebra, and let $\delta_1,\delta_2: A\to A\ot A$ be derivations. 

\begin{zlist}
\item There exists a coassociative covariant derivation $\Delta:A\to A\ot A$ with respect to $(\delta_1,\delta_2)$ if and only if there exists $u\in A\ot A$ such that, for all $a\in A$,
\begin{subequations}\label{cov.unit}
\begin{gather}
(\delta_1 - \delta_2)(a) = au - ua, \label{cov.unit.a}\\
(\delta_1\ot \id - \id\ot\delta_1)\circ \delta_1(a) = u_{23}\delta_1(a)_{13}, \label{cov.unit.b}\\
(\delta_1\ot \id - \id\ot\delta_1)(u) = u_{23} u_{13} - u_{12}u_{23}, \label{cov.unit.c}
\end{gather}
\end{subequations}
where the leg-numbering notation \eqref{legs} is used. In this case
\begin{equation}\label{copr.u}
\Delta(a) = ua + \delta_1(a) = au + \delta_2(a).
\end{equation}
\item The constant subalgebra of $(A,\delta_1,\delta_2, \Delta)$ is the maximal subalgebra of $A$ over which $\Delta$ is bilinear. 
\end{zlist}
\end{proposition}
\begin{proof}
(1) Let $\Delta: A\to A\ot A$ be a $\KK$-linear map and set $u=\Delta(1)$. Then $\Delta$ is a covariant derivation if and only if, for all $a\in A$,
$$
\Delta(a) = ua + \delta_1(a) = au + \delta_2(a),
$$
which is equivalent to \eqref{cov.unit.a} and necessarily includes \eqref{copr.u}. The coassociativity of $\Delta$ at $a=1$ is equivalent to \eqref{cov.unit.c}. Writing down the coassociativity condition for $\Delta$ at general $a\in A$, and using the covariant derivation property and \eqref{cov.unit.c}, one finds that this condition is equivalent to \eqref{cov.unit.b}. 

(2)  Since $\Delta$ is a covariant derivation, it is bilinear over $C(A)$.  Conversely, if $b\in A$ is such that $\Delta(ab) = \Delta(a)b$ and $\Delta(ba) = b \Delta(a)$, for all $a\in A$, then
$$
\delta_1(b) = \Delta(b) - ub = \Delta(1)b -ub =0.
$$
 by \eqref{copr.u}. Similarly, $\delta_2(b) =0$.
\end{proof}

\begin{corollary}\label{cor.unital}
$(A,\delta_1,\delta_2,\Delta)$ is a unital covariant bialgebra if and only if, for all $a\in A$,
\begin{subequations}\label{cov.unit.unit}
\begin{gather}
(\delta_1 - \delta_2)(a) = a\ot 1 - 1\ot a, \label{cov.unit.unit.a}\\
(\delta_1\ot \id - \id\ot\delta_1)\circ \delta_1(a) = \delta_1(a)_{13}. \label{cov.unit.unit.b}
\end{gather}
\end{subequations}
and 
\begin{equation}\label{copr.1}
\Delta(a) = 1\ot a + \delta_1(a). 
\end{equation}

A unital covariant bialgebra is connected.
\end{corollary}
\begin{proof}
If $(A,\delta_1,\delta_2,\Delta)$ is to be a unital covariant bialgebra, then $u=1\ot 1$, and equations \eqref{cov.unit} reduce to \eqref{cov.unit.unit}, while \eqref{copr.u} becomes \eqref{copr.1}. 

Obviously, $\KK 1\subseteq C(A)$. Conversely, if $a\in \ker\delta_2(a)$, then $\delta_1(a) =a\ot 1 -1\ot a$ by \eqref{cov.unit.unit.a}. Hence $a\in \KK 1$. Similarly, \eqref{cov.unit.unit.a} yields that $\delta_1(a) =0$ implies $a\in \KK 1$. Therefore, $(A,\delta_1,\delta_2,\Delta)$ is connected.
\end{proof}

\begin{remark}\label{rem.unital}
It is worth pointing out that if $A$ is a unital algebra, then in a  covariant bialgebra $(A,\delta_1,\delta_2,\Delta)$ both $\delta_2$ and $\Delta$ are fully determined by $\delta_1$ and $\Delta(1)$ through relation \eqref{cov.unit.a} and \eqref{copr.u}. (In a similar way $\delta_1$ and $\Delta$  are determined by $\delta_2$ and $\Delta(1)$.) This is a reason for $\delta_2$ not featuring in equations \eqref{cov.unit.b}-\eqref{cov.unit.c}. On the other hand, using \eqref{cov.unit.a} one can replace $\delta_1$ by $\delta_2$ in \eqref{cov.unit.b}-\eqref{cov.unit.c} and thus obtain the equivalent conditions
\begin{subequations}\label{cov.unit.2}
\begin{gather}
( \id\ot\delta_2 - \delta_2\ot \id )\circ \delta_2(a) = \delta_2(a)_{13}u_{12}, \label{cov.unit.2b}\\
(\delta_2\ot \id - \id\ot\delta_2)(u) = u_{12} u_{23} - u_{13}u_{12}, \label{cov.unit.2c}
\end{gather}
\end{subequations}

In the case of a unital covariant bialgebra $(A,\delta_1,\delta_2,\Delta)$, the form of any two of $\delta_1,\delta_2,\Delta$ is fully determined by the third one.
\end{remark}

\begin{example}\label{ex.dual}
Let $A= \KK[\zeta]/\langle\zeta^2\rangle$ be the algebra of dual numbers and consider the map 
$$
\delta_1: A\to A\ot A, \qquad 1\mapsto 0, \quad \zeta\mapsto \zeta\ot \zeta.
$$
One easily checks that $\delta_1$ is a derivation, and since $A$ is a unital algebra, in the view of Proposition~\ref{prop.cov.unit} and Remark~\ref{rem.unital}, we need to find $u\in A\ot A$ that satisfies \eqref{cov.unit.b}-\eqref{cov.unit.c}. Set
$$
u = a1\ot\zeta + b\zeta\ot\zeta + c1\ot 1 + d\zeta\ot 1.
$$
Since $(\delta_1\ot \id - \id\ot\delta_1)\circ \delta_1 =0$, equation \eqref{cov.unit.b} is equivalent to 
$$
(a1\ot1\ot \zeta + b1\ot \zeta\ot\zeta + c1\ot 1\ot 1 + d1\ot \zeta\ot 1)\zeta\ot 1\ot \zeta =0,
$$
which is satisfied provided $c=d=0$. Then \eqref{cov.unit.c} becomes
$$
-a^2 1\ot\zeta\ot\zeta - ab\ \zeta\ot\zeta\ot\zeta = -a 1\ot\zeta\ot\zeta,
$$
so that $a^2 =a$ and $ab=0$. Therefore,
$$
u = 1\ot \zeta \quad \mbox{or}\quad u = b\zeta\ot\zeta, \quad b \in \KK.
$$
In the first case,
$$
\Delta(1) = 1\ot \zeta, \qquad \Delta(\zeta) = \zeta\ot\zeta .
$$
This covariant bialgebra is left-connected but it is not right-connected. In the second case,
$$
\Delta(1) = b\zeta\ot\zeta, \qquad \Delta(\zeta) = (1+b) \zeta\ot\zeta,
$$
and the bialgebra is connected. Only when  $b=0$, $\Delta$ is a derivation, thus making $A$ into an infinitesimal bialgebra in this case; cf.\ \cite[Example~2.3.6]{Agu:inf}. 
\end{example}

\begin{remark}\label{rem.counit}
Dually to unital covariant bialgebras one can consider counital covariant bialgebras. If $A$ is a counital coalgebra with a counit $\eps$, and $(A,\delta_1,\delta_2,\Delta)$ is a covariant bialgebra, then the covariant derivation property of $\Delta$ yields, for all $a,b\in A$,
\begin{subequations}\label{cov.counit}
\begin{gather}
(\eps\ot \id)(a\delta_1(b)) = 0 , \qquad (\id\ot \eps)(\Delta(a) b +a\delta_1(b)) = ab \label{cov.counit.a}\\
(\id\ot \eps) (\delta_2(a)b) = 0, \qquad (\eps\ot \id)(a\Delta(b)  +\delta_2(a)b) = ab .
\end{gather}
\end{subequations}
If, in addition it is assumed that $\eps$ is a multiplicative map (a condition dual to that of $\Delta(1) =1\ot 1$ in case of the unital bialgebras), in which case one feels justified in talking about {\em a counital covariant bialgebra}, equations \eqref{cov.counit} reduce to 
\begin{subequations}\label{cov.counit.counit}
\begin{gather}
(\eps\ot \id)\circ \delta_1 = 0, \qquad  a\left(\left(\id\ot \eps\right) \circ \delta_1(b) -b +\eps(b)\right) = 0 \\
(\id\ot \eps) \circ \delta_2 =0, \qquad \left(\left(\eps\ot \id\right)\circ \delta_2(a) +a - \eps(a)\right) b =  0 ,
\end{gather}
\end{subequations}
for all $a,b\in A$. In contrast to infinitesimal bialgebras \cite[Remark~2.2]{Agu:inf}, there exist non-trivial covariant bialgebras that are both unital and counital. An example of such a bialgebra is described in Example~\ref{ex.counital}. 
\end{remark}

Next a characterisation of covariant bialgebras with inner derivations and coproducts is given.

\begin{proposition}\label{prop.inner}
Let $A$ be an algebra and $r,s\in A\ot A$. Define $\delta_r, \delta_s, \Delta : A\to A\ot A$ by
\begin{equation}\label{inner}
\delta_r(a) = ar - ra, \qquad \delta_s(a) = as-sa, \qquad \Delta(a) = ar-sa.
\end{equation}
Then $(A,\delta_r, \delta_s, \Delta)$ is a covariant bialgebra if and only if, for all $a\in A$,
\begin{equation}\label{inner.YB}
a(r_{13}r_{12} - r_{12}r_{23} +s_{23}r_{13}) = 
(s_{13}r_{12} - s_{12}s_{23} +s_{23}s_{13})a.
\end{equation}
\end{proposition}
\begin{proof}
Clearly $\delta_r, \delta_s$ are (inner) derivations and one easily checks that $\Delta$ is a covariant derivation with respect to  $(\delta_r, \delta_s)$. Writing $r = \sum r\su 1\ot r\su 2$, $s=\sum s\su 1\ot s\su 2$, one can compute, for all $a\in A$,
\begin{eqnarray*}
(\id\ot \Delta)\circ \Delta(a) &=& \sum ar\su 1\ot \Delta(r\su 2) - \sum s\su 1\ot \Delta(s\su 2a)\\
&=& \sum ar\su 1\ot r\su 2r - \sum ar\su 1\ot sr\su 2 - sar + \sum s\su 1\ot s s\su 2a\\
&=& ar_{12}r_{23} - as_{23}r_{13} -sar + s_{23}s_{13}a,
\end{eqnarray*}
and 
\begin{eqnarray*}
(\Delta \ot \id)\circ \Delta(a) &=& \sum  \Delta(ar\su 1)\ot r\su 2 - \sum  \Delta(s\su 1)\ot s\su 2a\\
&=& \sum ar\su 1r\ot r\su 2 - sar - \sum s\su 1r\ot s\su 2a + \sum s\su 1\ot s\su 2sa\\
&=& ar_{13}r_{12} - sar  -  s_{13}r_{12}a + s_{12}s_{23}a.
\end{eqnarray*}
Therefore, $\Delta$ is coassociative if and only if condition \eqref{inner.YB} is satisfied.
\end{proof}

\begin{remark} \label{rem.inner}
Note that in the case of a non-degenerate algebra $A$, the form of $\Delta$ in \eqref{inner} specifies the forms of $\delta_r$ and $\delta_s$. Suppose that $\Delta(a) = ra-sa$ is a $(\delta_1,\delta_2)$-covariant derivation. Then, for all $a,b\in A$,
$$
a(br -rb) = a\delta_1(b) \quad \mbox{and} \quad (as-sa)b = \delta_2(a)b.
$$
Since the product in $A$ is non-degenerate, these equations imply that $\delta_1 = \delta_r$ and $\delta_2 = \delta_s$.
\end{remark}

Immediately from Proposition~\ref{prop.inner} one deduces the existence of covariant bialgebras affiliated with  associative Yang-Baxter pairs. 
\begin{corollary}\label{cor.inner}
If $(r,s)$ is an associative Yang-Baxter pair on $A$ and $\delta_r, \delta_s,\Delta$ are defined by \eqref{inner}, then $(A,\delta_r, \delta_s, \Delta)$ is a covariant bialgebra. In this case,   $(A,\delta_r, \delta_s, \Delta)$ is called a {\em quasitriangular covariant bialgebra.}

Furthermore, if $f:A\to B$ is an algebra map and $(r^f,s^f)$ is the associative Yang-Baxter pair induced by $f$ as in \eqref{rfsf}, then $f$ is a morphism of quasitriangular covariant bialgebras from the one associated to $(r,s)$ to the one associated to $(r^f,s^f)$.
\end{corollary}
\begin{proof}
The first statement follows immediately from Proposition~\ref{prop.inner}. The second is a consequence of the definition of the coproduct associated to an associative Yang-Baxter pair and the multiplicativity of $f$.
\end{proof}
\begin{lemma}\label{lem.quasi.YB}
Let $A$ be an algebra, $r,s\in A\ot A$ and $\delta_r, \delta_s,\Delta$ be defined by \eqref{inner}. Then $(A,\delta_r, \delta_s, \Delta)$ is a quasitriangular covariant bialgebra if and only if
\begin{equation}\label{quasi.YB}
(\id\ot \Delta)(r) = r_{13}r_{12} \qquad \mbox{and} \qquad (\Delta\ot \id) (s) = -s_{23}s_{13}.
\end{equation}
\end{lemma}
\begin{proof}
Using the definition of $\Delta$ in \eqref{inner} one easily checks that
$$
(\id\ot \Delta)(r) = r_{12}r_{23} - s_{23}r_{13}.
$$
Hence the first of equations \eqref{quasi.YB} is equivalent to \eqref{rs.r}. In a similar way, the second of equations  \eqref{quasi.YB} is equivalent to \eqref{rs.s}.
\end{proof}
\begin{lemma}\label{lem.quasi}
Unital quasitriangular covariant bialgebra structures on $A$ are determined by $r\in A\ot A$ such that
\begin{equation}\label{quasi.uni}
r_{13} = r_{13}r_{12} - r_{12}r_{23} +r_{23}r_{13},
\end{equation}
or equivalently, 
\begin{equation}\label{quasi.YB.uni}
(\id\ot \Delta)(r) = r_{13}r_{12} \qquad \mbox{and} \qquad (\Delta\ot \id) (r) = -r_{23}r_{13} +r_{23}+ r_{13} ,
\end{equation}
where $\Delta(a) = 1\ot a + ar - ra$.
\end{lemma}
\begin{proof}
Quasitriangular covariant bialgebra $(A,\delta_r, \delta_s, \Delta)$ is unital if and only if $r-s=1\ot 1$. In this case equations \eqref{rs} reduce to \eqref{lem.quasi}, while \eqref{quasi.YB} become \eqref{quasi.YB.uni}.
\end{proof}
\begin{example}\label{ex.quasi.uni}
Suppose $e\in A$ is an idempotent and let $\kappa \in \{0,1\}\subset \KK$. Then  $r_e=\kappa 1\ot e+ (1-\kappa)e\ot 1$ solves equation \eqref{quasi.uni}. Hence $(r_e,s_e)$ with $s_e=\kappa 1\ot e+ (1-\kappa)e\ot 1 -1\ot 1$ is an associative Yang-Baxter pair, and
\begin{equation}\label{deltae}
\Delta_e(a) = \kappa (a\ot e - 1\ot ea) ) + (1-\kappa)(ae\ot 1 - e\ot a) +1\ot a.
\end{equation}
 The solutions to the Rota-Baxter system come out as
\begin{equation}\label{Re}
R_e(a) = ea + \kappa (ae -ea), \qquad S_e(a) = (e-1)a+ \kappa (ae -ea).
\end{equation}
The associative and pre-Lie algebra structures on $A$ arising from Corollary~\ref{cor.RB} are
$$
a*_eb = eab +a(e-1)b + \kappa (abe - eab), \quad a\bullet_e b = eab - b(e-1)a + \kappa (aeb +bea - eab -bae).
$$
Taking $e=0,1\in A$, one obtains from \eqref{deltae} 
$$
\Delta_0(a) = 1\ot a, \qquad \Delta_1 = a\ot 1.
$$
$(A,0,\delta_U, \Delta_0)$ and $(A,\delta_U, 0, \Delta_1)$ are unique unital covariant bialgebras associated to the universal derivation $\delta_U: a\mapsto a\ot 1-1\ot a$. 
Seen from a different perspective, these are unique unital covariant bialgebras in which the comultiplication is either right (the case of $\Delta_0$) or left (the case of $\Delta_1$) $A$-linear. 

For a concrete solution to the Rota-Baxter system of the type \eqref{Re}, consider $A=M_2(\KK)$ and $e=\begin{pmatrix} 1 & 0\cr 0 & 0\end{pmatrix}$. Then 
$$
R_e\begin{pmatrix} a & b\cr c & d\end{pmatrix} = \begin{pmatrix} a & (1-\kappa) b \cr \kappa c &0\end{pmatrix},
\qquad S_e\begin{pmatrix} a & b\cr c & d\end{pmatrix} =\begin{pmatrix} 0 & -\kappa b \cr (\kappa -1)c & -d\end{pmatrix}
$$
\end{example}

\begin{example}\label{ex.counital}
The construction of Example~\ref{ex.quasi.uni} can be employed to produce an example of a quasitriangular  unital  and counital covariant bialgebra. Let 
$$
A=\KK[a,e]/\langle e^2-e ,\ ae\rangle.
$$
For $\kappa=1$, the unital coproduct \eqref{deltae} comes out as
\begin{equation}\label{delta.couni}
\Delta_e(a^n) =  a^n\ot e  +1\ot a^n, \qquad  
\Delta_e(e) = e\ot e .
\end{equation}
The coproduct $\Delta_e$ admits a multiplicative counit, 
\begin{equation}\label{eps.couni}
\eps(1) = \eps(e) = 1, \qquad \eps(a^n) = 0.
\end{equation}
Thus $A$ with coproduct \eqref{delta.couni} and counit \eqref{eps.couni} is a quasitriangular unital and counital covariant bialgebra.
\end{example}

\subsection{Covariant modules}

Hopf modules, i.e.\ vector spaces with compatible action and coaction of a bialgebra, form the representation category of bialgebras. In a similar manner, the representation category of covariant bialgebras is provided by covariant modules.

\begin{definition}\label{def.mod}
Let $(A,\delta_1,\delta_2,\Delta)$ be a covariant bialgebra. A right $A$-module and a right $A$-comodule $M$ is said to be a {\em right covariant $A$-module} provided the coaction is a right covariant $\delta_1$-derivation. Symmetrically, a left $A$-module and a left $A$-comodule $N$ is said to be a {\em left covariant $A$-module} provided the coaction is a left covariant $\delta_2$-derivation. A morphism of covariant modules is a map that is both $A$-linear and $A$-colinear. The category of right covariant $A$-modules is denoted by $\M^A_A$. If $A$ has an identity, then the full subcategory  of $\M^A_A$ consisting of unital modules is denoted by $\uM^A_A$. If $A$ has a counit, then the full subcategory of $\M^A_A$ consisting of counital modules is denoted by $\cM^A_A$. The category of right covariant $A$-modules that are both unital and counital is denoted by $\cuM^A_A$.
\end{definition}

By the obvious left-right symmetry of Definition~\eqref{def.mod} whatever is said about right covariant modules can equally well be said about left covariant modules.

\begin{example}\label{ex.module}
$A$ is both a left and right covariant module over $(A,\delta_1,\delta_2,\Delta)$ via multiplication and comultiplication. Consequently, for any subalgebra $B$ of the constant algebra $C(A)$ (see Definition~\ref{def.connected}) the assignment
\begin{equation}\label{ind}
V\mapsto V\ot_B A, \quad (v\ot_B a)b = v\ot_B ab, \quad v\ot_B a\mapsto v\ot_B \Delta(a), \qquad v\in V,\ a,b\in A
\end{equation}
defines a functor $F_B:\M_B \to  \M_A^A$  from the category of right $B$-modules to the category of right covariant $A$-modules. On morphisms $F_B$ is defined by $f\mapsto f\ot_B \id$.

Clearly, if $A$ has an identity, then the image of $F_B$ is in $\uM_A^A$, and if $A$ has counit, then the image of  $F_B$ is in $\cM_A^A$. 

The functors $F_B$ defined in this example are called {\em  free right covariant  module functors relative to $B$}. Symmetrically, free left  covariant module functors are defined.
\end{example}

\begin{proposition}\label{prop.adjoint}
Let $A$ be a unital algebra and  $(A,\delta_1,\delta_2, \Delta)$ be a covariant bialgebra, let $B\subseteq C(A)$ be a subalgebra, and set $u:= \Delta(1)$. For any right covariant $A$-module $M$ with coaction $\roM$, the {\em coinvariants} are defined by
\begin{equation}\label{coin}
M^{co A}:= \{m\in M\; |\; \roM(m) = mu\}.
\end{equation}
\begin{zlist}
\item The functor given on objects by
\begin{equation}\label{coin.fun}
G_B: \uM_A^A \to \M_B, \qquad M\mapsto M^{coA},
\end{equation}
and as identity on morphisms, is the right adjoint to the free right covariant  module functor $F_B:\M_B \to  \uM_A^A$ in Example~\ref{ex.module}.
\item If $A$ is  flat as a right $B$-module and $\ker \delta_2 = A^{coA} \subseteq B$, then $F_B$ is a full and faithful functor.
\end{zlist} 
\end{proposition}
\begin{proof}
(1) Note that the functor $G_B$ is well-defined, since, for all $b\in B$ and $m\in M^{coA}$,
$$
\roM(mb) = \roM(m)b + \delta_1(b) = mub = mbu,
$$
where the second equality follows by the fact that $B\subseteq C(A)$, and the third one is a consequence of \eqref{cov.unit.a}. 

The unit and counit of the adjunction are defined by, for all $V\in \M_B$ and $M\in \uM^A_A$, 
\begin{equation}\label{un.co}
\eta_V : V\to (V\ot_A A)^{coA}, \quad v\mapsto v\ot_B 1, \qquad \varphi_M: M^{coA}\ot_B A\to M, \quad m\ot_B a\mapsto ma.
\end{equation}
Clearly, $\eta_V$ is a well-defined right $B$-linear map natural in $V$. It is equally clear that $\varphi_M$ is a right $A$-linear map. To check the right $A$-colinearity of $\varphi_M$, take any $m\in M^{coA}$ and $a\in A$ and compute
\begin{eqnarray*}
\roM(\varphi_M(m\ot_B a)) &=& \roM(ma) = \roM(m)a + m\delta_1(a)\\
&=& mua + m\delta_1(a) = m\Delta(a) = (\varphi_M\ot \id)\circ (\id \ot_B \Delta)(m\ot_B a).
\end{eqnarray*}
The second equality follows by the definition of a right covariant $A$-module, the third one is a consequence of the fact that $m$ is a coinvariant element of $M$, and the fourth equality follows by \eqref{copr.u}. The naturality of $\varphi_M$ is obvious. Finally, the triangle equalities that a unit and a counit of an adjunction are required to satisfy follow by the unitality of the algebra $A$ and its unital modules.

(2) The equality $\ker\delta_2 = A^{coA}$ in the hypothesis follows by \eqref{copr.u}. Take a right $B$-module $V$ and assume that $\sum_i v_i\ot_B a_i \in (V\ot_B A)^{coA}$. Then 
$$
\sum_i v_i \ot_B \Delta(a_i) = \sum_i v_i\ot_B a_i u,
$$ 
and hence, in view of the second of equations \eqref{copr.u},
$$
\sum_i v_i \ot_B \delta_2(a_i) = 0.
$$ 
By the flatness of $A$ as a right $B$-module, $\delta_2(a_i) =0$, and since $\ker\delta_2 \subseteq B$, $a_i\in B$. Thus in this case we can define
\begin{equation}\label{inverse}
\eta^{-1}_V (\sum_i v_i\ot_B a_i ):=\sum_i v_i a_i . 
\end{equation}
One easily checks that $\eta^{-1}_V$ is the (natural) inverse to the unit of adjunction $\eta_V$. Hence $F_B$ is a full and faithful functor.
\end{proof}

\begin{remark}\label{rem.ff.coinv}
Using standard methods of Hopf-Galois theory one easily finds a sufficient condition for the coinvariants functor $G_B$ to be a full and faithful functor. Start with a covariant bialgebra
 $(A,\delta_1,\delta_2, \Delta)$, choose $B\subseteq C(A)$ and assume that $A$ is a unital algebra, admits a {\em right} counit,  and that $A$ is flat as a left $B$-module. If there exists a $\KK$-linear map $\tau: A\to A\ot_BA$, written on elements as $\tau(a) = \sum a\suc 1\ot_B a\suc 2$, such that, for all $a\in A$,
 \begin{subequations}\label{trans}
 \begin{gather}
 \sum a\sw 1 \tau(a\sw 2) = 1\ot_B a, \label{trans.a}\\
 (\delta_1\ot_B\id)\circ \tau (a) = \sum a\suc 1u\ot_B a\suc 2 - 1\ot 1\ot_B a, \label{trans.b}\\
 \sum a\suc 1 a\suc 2 = \eps(a) 1,\label{trans.c}
 \end{gather}
 \end{subequations}
 where $u = \Delta(1)$, then the counit $\varphi$ of the adjunction $F_B \dashv G_B$ has a natural inverse, for all $M\in \cuM_A^A$,
 $$
 \varphi_M^{-1}: M\to M^{co A}\ot_B A, \qquad m\mapsto \sum m\sw 0\tau(m\sw 1),
 $$
where the standard Sweedler notation for coactions $\roM(m) = \sum m\sw 0\ot m\sw 1$ is used. Therefore, $G_B$ is a full and faithful functor in this case. If in addition the hypothesis of Proposition~\ref{prop.adjoint}(2) is satisfied (in which case, condition \eqref{trans.c} implies \eqref{trans.a}), then the free covariant module functor $F_B: \M_B \to \cuM^A_A$ is an equivalence.
 
While the hypothesis of Proposition~\ref{prop.adjoint}(2) is easily fulfilled (for example, it holds for $B=\KK$ and all right-connected covariant bialgebras, thus in particular for all unital covariant bialgebras), examples of covariant bialgebras admitting function $\tau$ satisfying conditions \eqref{trans} (apart from the obvious trivial case $(\KK, 0,0,\id)$) seem to be hard to come by. Thus: whether a definition of a {\em covariant Hopf algebra} modelled on the fundamental theorem of Hopf modules which asserts that a bialgebra is a Hopf algebra if and only if the free Hopf module functor from the category of vector spaces to right (or left) Hopf modules is an equivalence, and in particular insisting  on existence of $\tau$ satisfying condition \eqref{trans}, is appropriate in the realm of covariant bialgebras remains to be seen. The definition of an {\em infinitesimal Hopf algebra} proposed in \cite{Agu:inf} is  motivated by a different viewpoint.
\end{remark}

In the case of quasitriangular covariant bialgebras, every $A$-module is a covariant module in a natural way (cf.\ \cite[Proposition~153, p.\ 339]{CaeMil:Fro}).
\begin{proposition}\label{prop.quasi.module}
Let $(A,\delta_r, \delta_s, \Delta)$ be a quasitriangular covariant bialgebra associated to an  associative Yang-Baxter pair $(r,s)$. Then every right $A$-module $M$ is a right covariant $A$-module with the coaction
\begin{equation}\label{mod.r}
\roM: M\to M\ot A, \qquad m\mapsto mr.
\end{equation}
Similarly, every left $A$-module $N$ is a left covariant $A$-module with the coaction
\begin{equation}\label{mod.l}
\roN: N\to A\ot N, \qquad n\mapsto -sm.
\end{equation}
\end{proposition}
\begin{proof}
Observe that, for all $m\in M$,
$$
(\roM\ot \id)\circ\roM (m) = mr_{13}r_{12},
$$
and this is equal to $(\id\ot \Delta)\circ\roM (m)$ by Lemma~\ref{lem.quasi} and the definition of $\roM$ in \eqref{mod.r}. The second statement is proven in a similar way.
\end{proof}

\subsection{Extensions}\label{sec.ext}

Unlike the definition of a bialgebra, the definition of a covariant bialgebra does not use a switch or a twist map between the factors in a tensor product. Therefore, it can be verbatim transferred to any monoidal category enriched over vector spaces (or Abelian groups). The same is true about the definition of covariant modules. As a special case, one can consider the category of bimodules over a unital associative ring $B$ with the tensor product over $B$ as a monoidal structure. Following a long standing Hopf algebra tradition, a covariant bialgebra in this category might be  called a {\em covariant bialgebroid} and comprises  a $B$-bimodule $A$ together with a $B$-bilinear associative multiplication $\mu: A\ot_B A \to A$ (i.e.\ $(A,\mu)$ is a {\em non-unital $B$-ring}), two $B$-bilinear $A$-derivations $\delta_1, \delta_2: A\to A\ot_B A$ and a $B$-bilinear coassociative comultiplication $\Delta: A\to A\ot_B A$ that is required to be a $(\delta_1,\delta_2)$-covariant derivation. 

If a monoidal category has no braiding (as is the case, for example in the category of $B$-bimodules), equations \eqref{rs} no longer make sense in general. Still, even in such cases some additional conditions on $r$ and $s$ can be put, so that an associative Yang-Baxter pair and the corresponding Rota-Baxter system are formed. For example, in the case of $B$-bimodules, for any $B$-bimodule $M$, denote the centraliser of $B$ in $M$ by 
$$
M^B:= \{m\in \; |\; \mbox{for all $b\in B$},\; mb =bm\}.
$$
For any $r,s\in (A\ot_BA)^B$,  equations \eqref{rs} can be formed, maps \eqref{inner} are $B$-bilinear and $\Delta$ is coassociative provided $r,s$ solve \eqref{rs}. Furthermore, Proposition~\ref{prop.inner}  and Proposition~\ref{prop.quasi.module} (with the word bialgebra replaced by $B$-bialgebroid) remain valid for all $r,s\in (A\ot_BA)^B$.

\section{Twisted Rota-Baxter operators}
\label{sec.twiR-B}
To allow for the interpretation of Jackson's $q$-integral as a Rota-Baxter-type operator  the introduction of an endomorphism twist into the definition of the Rota-Baxter algebra seems to be needed. This is prompted by the fact that the $q$-derivation to which Jackson's integral is an inverse operation is itself an endomorphism (in fact automorphism) twisted derivation. In this section we propose a suitable twisting of a Rota-Baxter operator and show that every twisted Rota-Baxter operator leads to a Rota-Baxter system on the algebra on which it operates.
\begin{definition}\label{def.RBtwist}
Let $A$ be an associative algebra and let  $\sigma:A\to A$ be a multiplicative map. A $\KK$-linear operator $R:A\to A$ is called a {\em $\sigma$-twisted Rota-Baxter operator} if, for all $a,b\in A$,
\begin{equation}\label{twi.RB}
R(a)R(b) = R(R(a) b + a\ R^\sigma(b)),
\end{equation}
where $R^\sigma:= \sigma\circ R$.
\end{definition}

One should be made aware that the notion introduced in Definition~\ref{def.RBtwist} is different from that of \cite[Section~3]{Uch:ana}, which uses a cocycle rather than algebra homomorphism for twisting.

\begin{lemma}\label{lem.twist}
Let $R$ be a $\sigma$-twisted Rota-Baxter operator on $A$. Then $(A,R,R^\sigma)$ is a Rota-Baxter system. 
\end{lemma}
\begin{proof}
Equation \eqref{twi.RB} is equivalent to \eqref{RB.r} with $S=R^\sigma$. Applying $\sigma$ to \eqref{twi.RB} and using the multiplicativity of $\sigma$ one obtains \eqref{RB.s} with $S=R^\sigma$.
\end{proof}

The introduction of twisted Rota-Baxter operators is motivated by the following
\begin{example}\label{ex.Jackson}
Assume that $\KK$ is a field of characteristic zero, and let $q\in \KK$ be a non-zero number that is not a root of unity. For all integers $n$ denote by $[n]_q$ the $q$-integers,
$$
[n]_q = \frac{1-q^n}{1-q}.
$$
Let $A=\KK[x]$ a polynomial ring and define $\KK$-linear operators $\sigma, J: A\to A$ by
\begin{equation}\label{jac}
\sigma(x^n) = q^n x^n, \qquad J(x^n) = \frac{1}{[n+1]_q} x^{n+1}.
\end{equation}
The operator $J$ is the Jackson $q$-integral. Clearly, $\sigma$ is an algebra map and 
$$
J^\sigma(x^n) := \sigma\circ J(x^n) = \frac{q^{n+1}}{[n+1]_q} x^{n+1},
$$
so that
$$
J(J(x^n)x^m + x^n J^\sigma(x^m)) = \frac{[m+n+2]_q}{[m+1]_q[n+1]_q} J(x^{m+n+1}) = J(x^n)J(x^m).
$$
Therefore, the Jackson $q$-integral is a twisted Rota-Baxter operator. In view of Corollary~\ref{cor.RB}, $\KK[x]$ has an associative product
$$
x^n*x^m = \frac{[m+n+2]_q}{[m+1]_q[n+1]_q} x^{m+n+1},
$$
while 
$$
x^n\bullet x^m = (1-q)x^{m+n+1},
$$
defines a pre-Lie algebra structure on $\KK[x]$.
\end{example}

\begin{remark}\label{rem.Jackson}
The fact that the Jackson integral is a solution to the twisted Rota-Baxter equation was observed in \cite{EbrMan:twi}. This observation led the authors to introduce the notion of a twisted dendriform algebra \cite[Definition~1]{EbrMan:twi}.  One can easily prove that in fact, by taking into account the twist in the definition of one of the dendriform operations (as indicated, for example, by formula (41) in \cite{EbrMan:twi}), a twisted dendriform algebra can always be interpreted as a dendriform algebra.
\end{remark}

Differential Rota-Baxter algebras provide yet another example of a twisted Rota-Baxter operator.

\begin{example}\label{ex.diff}
Following \cite[Definition~1.1]{GuoKei:dif} a {\em differential Rota-Baxter algebra of weight $\lambda$} is a triple $(A,R,\partial)$ such that $(A,R)$ is a Rota-Baxter algebra of weight $\lambda$, the $\KK$-linear operator $\partial: A\to A$ satisfies the twisted Leibniz rule,
\begin{equation}\label{Leibniz.lambda}
\partial(ab) = \partial(a)b + a\partial(b) +\lambda\partial(a)\partial(b), \qquad \mbox{for all $a,b\in A$,}
\end{equation}
and
\begin{equation}\label{R.par}
\partial \circ R = \id.
\end{equation}
Define a $\KK$-linear operator 
\begin{equation}\label{sig}
\sigma_\lambda : A\to A, \qquad a\mapsto a +\lambda\partial(a).
\end{equation}
Then the twisted Leibniz rule \eqref{Leibniz.lambda} implies that $\sigma_\lambda$ is an algebra map, while \eqref{R.par} implies that 
$$
R^{\sigma_\lambda}(a) = R(a) + \lambda a.
$$
Hence $R$ is a $\sigma_\lambda$-twisted Rota-Baxter operator.
\end{example}

\begin{example}\label{ex.diff.twist}
Example~\ref{ex.diff} suggests the following generalisation of differential Rota-Baxter algebras. Recall that given an algebra endomorphism $\sigma: A\to$, and $\KK$-linear map $\partial:A\to A$ is called a {\em $\sigma$-derivation}, provided,
\begin{equation}\label{Leibniz.twist}
\partial(ab) = \partial(a)\sigma(b) + a\partial(b), \qquad \mbox{for all $a,b\in A$.}
\end{equation}
A  {\em twisted differential Rota-Baxter algebra}  is quadruple $(A,\sigma,\partial,R)$ consisting of an algebra $A$, an algebra endomorphism $\sigma: A\to A$, a $\sigma$-twisted derivation $\partial:A\to A$ and a $\sigma$-twisted Rota-Baxter operator $R:A\to A$ satisfying \eqref{R.par}.

In the context of the Jackson $q$-integral Example~\ref{ex.Jackson}, define $\partial: \KK[x]\to \KK[x]$ by $\partial(x^n) = [n]_qx^{n-1}$. Then $(\KK[x], \sigma, \partial, J)$, with $\sigma$ and $J$ given by \eqref{jac}, is a twisted differential Rota-Baxter algebra.
\end{example}

\section{Conclusions}
In this paper we proposed a modification of the notion of a Rota-Baxter algebra. A number of examples (general and specific) as well as an algebraic grounding for Rota-Baxter systems were presented. Whether  Rota-Baxter systems can find applications in the areas to which traditionally Rota-Baxter algebras are applied (probability, combinatorics, renormalisation of quantum field theories) remains to be seen. On the algebraic side, the author believes that covariant bialgebras deserve to be studied further, and that such a study  constitutes an interesting research avenue.

In a different direction, one can easily write the Lie-algebraic version of classical Yang-Baxter pairs,
\begin{equation}\label{rs.Lie}
[r_{12},r_{13}]  + [r_{12},r_{23}] +[s_{13},r_{23}] = 0, \qquad [s_{12},r_{13}] + [s_{12},s_{23}] +[s_{13},s_{23}] = 0, \end{equation}
Whether the system of equations \eqref{rs.Lie} (e.g.\ with Poisson brackets as Lie operations)
\begin{blist}
\item  can find usage in the Hamiltonian theory and lead to solutions of new and existing examples of integrable systems, 
\item can be quantised and lead to interesting algebraic structures,
\end{blist}
in the way similar to the classical Yang-Baxter equations, are open questions which the author believes are worth further study.

\end{document}